\documentclass[11pt,margin=1in,lmargin=1.25in]{amsart}
\usepackage{amsfonts, amsmath, amssymb, color}
\usepackage[pagewise]{lineno}
\usepackage{amsthm}
\usepackage{booktabs}
\usepackage{float}
\usepackage{hhline}
\usepackage{lipsum}
\usepackage{lscape}
\usepackage{subfig}
\usepackage{textcomp}
\usepackage{tikz}
\usetikzlibrary{decorations.pathmorphing,shapes,arrows,positioning}
\usepackage{xcolor}
\usepackage{young}
\usepackage{youngtab}
\usepackage{ytableau}
\allowdisplaybreaks[4]
\theoremstyle{plain}
\newtheorem{thm}{Theorem}[section]

\newtheorem{prop}[thm]{Proposition}
\newtheorem{cor}[thm]{Corollary}

\theoremstyle{definition}
\newtheorem{defn}[thm]{Definition}
\newtheorem{exmp}[thm]{Example}

\newcommand{\la}{\lambda}

\newcommand{\tabincell}[2]{\begin{tabular}{@{}#1@{}}#2\end{tabular}}  
\numberwithin{equation}{section} \errorcontextlines=0

\begin{document}
\title{Spin Kostka polynomials and vertex operators}
\author{Naihuan Jing}
\address{Department of Mathematics, North Carolina State University, Raleigh, NC 27695, USA}
\email{jing@ncsu.edu}
\author{Ning Liu}
\address{School of Mathematics, South China University of Technology,
Guangzhou, Guangdong 510640, China}
\email{mathliu123@outlook.com}

\maketitle

\begin{abstract}
An algebraic iterative formula for the spin Kostka-Foulkes polynomial $K^-_{\xi\mu}(t)$ is given using vertex operator realizations of Hall-Littlewood symmetric functions and Schur's Q-functions. Based on the operational formula, more favorable properties are obtained parallel to the Kostka polynomial. In particular, we obtain some formulae for the number of (unshifted) marked tableaux. As an application,
we confirmed a conjecture of Aokage on the expansion of the Schur $P$-function in terms of Schur functions.
Tables of $K^-_{\xi\mu}(t)$ for $|\xi|\leq6$ are listed.
\end{abstract}

\section{Introduction}
The Hall-Littlewood symmetric functions $P_{\mu}(x;t)$ and the Kostka-Foulkes polynomials $K_{\lambda\mu}(t)$ both have played an active role in algebraic combinatorics and representation theory. On one hand, the Hall-Littlewood symmetric functions $P_{\mu}(x;t)$ are certain deformation of the Schur functions $s_{\lambda}(x)$ and the
Kostka-Foulkes polynomials $K_{\lambda\mu}(t)$ are the transition coefficients between the two bases. On the other hand, $K_{\lambda\mu}(t)$ have the following representation theoretic interpretation. Let
$\mathfrak{B}_{\mu}$ be the variety of flags preserved by a nilpotent matrix with Jordan block of shape $\mu$. The cohomology group $H^\bullet(\mathfrak{B}_{\mu})$ affords a graded $\mathfrak{S}_n$-module structure. Set
\begin{align*}
C_{\lambda\mu}(t)=\sum\limits_{i\geq 0}t^i (\text {dim Hom}_{\mathfrak{S}_n}(S^{\lambda}, H^{2i}(\mathfrak{B}_{\mu})),
\end{align*}
where $S^{\lambda}$ denotes the Specht module of $\mathfrak{S}_n$ associated with $\lambda$. Garsia and Procesi \cite{GP} proved that
\begin{align}
K_{\lambda\mu}(t)=C_{\lambda\mu}(t^{-1})t^{n(\mu)},
\end{align}
which confirms geometrically the positivity of the Kostka-Foulkes polynomials \cite{LS}.

Recently, Wan and Wang \cite{W} have introduced the spin
Kostka-Foulkes polynomials $K^-_{\xi\mu}(t)$
as the transition coefficients between the Hall-Littlewood functions $P_{\mu}(x;t)$ and Schur's $Q$-functions $Q_{\xi}$ with interesting representation theoretic interpretations. As is well-known,
the Schur Q-functions are indexed by strict partitions and were used by Schur \cite{S} in generalizing the Frobenius character formula for projective irreducible characters of the symmetric group $\mathfrak S_n$. Schur's Q-functions form
a distinguished basis in the subring of symmetric functions generated by $p_1, p_3, \cdots$. Yamaguchi \cite{Y} has shown that the category of irreducible $\mathfrak S_n$-supermodules is equivalent to that of supermodules of the Hecke-Clifford algebra $\mathcal H_n=\mathcal C_n\rtimes \mathbb C\mathfrak S_n$ and the irreducible objects $D^{\xi}$ are parametrized by strict partitions $\xi\in\mathcal{SP}_n$. Wan and Wang have shown that the spin Kostka polynomials admit the following interpretation \cite{W}:
\begin{align}
K^-_{\lambda\mu}(t)=2^{[l(\xi)/2]}C_{\xi\mu}(t^{-1})t^{n(\mu)},
\end{align}
and
\begin{align*}
C^-_{\xi\mu}(t)=\sum\limits_{i\geq 0}t^i (\text {dim Hom}_{\mathcal{H}_n}(D^{\xi}, \mathcal{C}_n\otimes H^{2i}(\mathfrak{B}_{\mu})).
\end{align*}

Let $\mathfrak q(n)$ be the queer Lie superalgebra containing the general linear Lie algebra $\mathfrak{gl}(n)$ as its even subalgebra. Sergeev \cite{Se} has shown that the irreducible 
$\mathfrak q(n)$-modules $V(\xi)$ are also parametrized by strict partitions $\xi\in\mathcal{SP}_n$.
It turns out that the
$q$-weight multiplicity $\gamma^-_{\xi\mu}(t)$ associated with the weight space $V(\xi)_{\mu}$ also appears
as the spin Kostka polynomial \cite{W}:
\begin{align}
K^-_{\lambda\mu}(t)=2^{[l(\xi)/2]}\gamma^-_{\xi\mu}(t).
\end{align}

The purpose of this paper is to give an operational algebraic formula for the spin Kostka-Foulkes polynomials $K^-_{\xi\mu}(t)$. The method we adopt is similar to that of \cite{BJ}, in which the vertex operator realizations of the Hall-Littlewood polynomials and Schur functions were employed. However, there is some subtlety in the spin situation.

In the usual vertex realization of Schur's Q-functions \cite{Jing91}, only the modes of odd indices (of the twisted Heisenberg algebra) were used in the definition. Should this vertex operator be employed, the commutation relations of its components with those of the vertex operator for the
Hall-Littlewood symmetric functions would have infinitely many terms in the quadratic relations.
To salvage the situation, we introduce a new vertex operator realization of the Schur Q-functions using a larger Heisenberg algebra graded by all integers
(see \eqref{e:schurQop}-\eqref{e:schurQop1}). The new
vertex operator realization enables us to get a finite quadratic relation between the operators realizing both the Hall-Littlewood and Schur Q-functions
and then the matrix coefficients express the spin Kostka polynomials.

As matrix coefficients, the spin Kostka-Foulkes polynomials can be computed
in general and exact formulas are obtained in some special cases. We also prove
a stability formula for the spin Kostka polynomials. We have clarified some questions regarding them (disproved the symmetric property, cf. Ex. \ref{e:counter}) and obtained counting formulas for the Stembridge coefficients \cite{S} between the
Schur P functions and Schur functions. As applications, we have confirmed a recent conjecture of Aokage and are
able to derive a tensor decomposition in the general situation.

The paper is organized as follows. In Section \ref{S:HLop} we recall the vertex operator realization of the Hall-Littlewood functions and give a new vertex operator construction of the Schur Q-functions, which is specifically tailored  for taming the commutation relation between the two vertex operators. In Section \ref{S:spinHL} we express
the spin Hall-Littlewood polynomials as matrix coefficients of vertex operators and derive
an iterative formula (see Theorem \ref{t:Recurrence Formula}).
Finally in Section \ref{S:MT} we use the iterative formulas to verify Aokage's conjecture on multiplicities of tensor products of spin modules, and a formula is also obtained for the general case.

\section{Vertex operator realization of Hall-Littlewood and Schur Q-functions}\label{S:HLop}

A partition (resp. strict partition) $\lambda=(\lambda_1,\lambda_2,\ldots)$, denoted $\lambda\vdash n$, is a weakly (resp. strictly) decreasing sequence of positive integers such that $\sum_i\la_i=n$. The sum
$|\lambda|=\sum_i\la_i$ is called the weight and the number $l(\lambda)$ of nonzero parts is called the length.
We also denote $\lambda\models n$ if $\lambda$ is a composition of $n$ when the pasts $\la_i$ are not necessarily ordered.
 The set of partitions (resp. strict partitions) of weight $n$ will be denoted by $\mathcal P_n$ (resp. $\mathcal{SP}_n$). The dominance order
 $\lambda\geq\mu$ is
 defined by  $|\lambda|=|\mu|$ and $\lambda_1+\cdots+\lambda_i\geq\mu_1+\cdots+\mu_i$ for each $i$.

Let $m_i$ be the multiplicity of $i$ in $\lambda$ and set $z_{\lambda}=\prod_{i\geq 1}i^{m_i(\lambda)}m_i(\lambda)!$, we define the parity $\varepsilon_{\lambda}=(-1)^{|\lambda|-l(\lambda)}$ and
\begin{align}\label{e:zt}
z_{\lambda}(t)&=\frac{z_{\lambda}}{\prod_{i\geq 1}(1-t^{\lambda_i})},\\
n(\lambda)&=\sum\limits_{i\geq1}(i-1)\lambda_i.
\end{align}

A partition $\la$ can be visualized by its Young diagram when $\la$ is identified with $\{(i,j)\in\mathbb{Z}^2\mid 1\leq i\leq l(\la), 1\leq j\leq \la_i\}$. To each cell $(i,j)\in \la$, we define its content $c_{ij}=j-i$ and hook length $h_{ij}=\la_i+\la^{'}_{j}-i-j+1$, where the partition $\la'=(\la_1', \ldots, \la_{\la_1}')$ is the dual partition of $\lambda$ obtained by reflecting the Young diagram of $\la$ along the diagonal.

In this paper, we use the t-integer $[n]=t^{n-1}+t^{n-2}+\cdots+t+1$. Similarly $[n]! = [n]\cdots[1]$, and
the Gauss $t$-binomial symbol $
\left[\begin{matrix}n\\k\end{matrix}\right]=\frac{[n]!}{[k]![n-k]!}
$.

Let $\Lambda_F$ be the ring of symmetric functions over $F=\mathbb Q(t)$, the field of rational functions in $t$. We also consider $\Lambda$ over the ring of integers. The space $\Lambda_F$ is graded and decomposes into a direct sum
\begin{align}\label{e:grade}
\Lambda_F=\bigoplus_{n=0}^{\infty} \Lambda_F^n.
\end{align}
where $\Lambda_F^n$ is the subspace of degree $n$, spanned the element $p_{\lambda}=p_{\lambda_1}p_{\lambda_2}\cdots p_{\lambda_l}$ with $|\lambda|=n$. Here
$p_r$ is the degree $r$ power sum symmetric function.

Let $\Gamma_{\mathbb{Q}}$ be the subring of $\Lambda_{\mathbb{Q}}$ generated by the $p_{2r-1}$, $r\in\mathbb N$.
\begin{align*}
\Gamma_{\mathbb{Q}}=\mathbb{Q}[p_r: r \quad\text{odd}].
\end{align*}
The Schur $Q$-functions $Q_{\xi},$ $\xi$ strict, form a $\mathbb{Q}$-basis of $\Gamma_{\mathbb Q}$ \cite{M}. $\Gamma$ is also a graded ring $\Gamma=\oplus_{n\geq0}\Gamma^{n},$ where $\Gamma^{n}=\Gamma\cap \Lambda^{n}.$

The space $\Lambda_F$ is equipped with the bilinear form $\langle\ , \ \rangle$ defined by
\begin{align}\label{e:form}
\langle p_{\lambda}, p_{\mu}\rangle=\delta_{\lambda\mu}z_{\lambda}(t).
\end{align}
As $\{z_{\lambda}(t)^{-1}p_{\lambda}\}$ is the dual basis of the power sum basis, the adjoint operator of the multiplication operator $p_n$
is the differential operator $p_n^* =\frac{n}{(1-t^n)}\frac{\partial}{\partial p_n}$ of degree $-n$.

We recall the vertex operator realization of the Hall-Littlewood symmetric functions \cite{Jing1}
and  construct a variant vertex operator for the Schur Q-function on the space $\Lambda_F$.
The \textit{vertex operators} $H(z)$ and its adjoint $H^*(z)$ are $t$-parametrized linear maps: $\Lambda_F\longrightarrow \Lambda_F[[z, z^{-1}]]=\Lambda_F\otimes F[z, z^{-1}]$ defined by
\begin{align}
\label{e:hallop}
H(z)&=\mbox{exp} \left( \sum\limits_{n\geq 1} \dfrac{1-t^{n}}{n}p_nz^{n} \right) \mbox{exp} \left( -\sum \limits_{n\geq 1} \frac{\partial}{\partial p_n}z^{-n} \right)\\ \notag
&=\sum_{n\in\mathbb Z}H_nz^{n},\\
H^*(z)&=\mbox{exp} \left(-\sum\limits_{n\geq 1} \dfrac{1-t^{n}}{n}p_nz^{n} \right) \mbox{exp} \left(\sum \limits_{n\geq 1} \frac{\partial}{\partial p_n}z^{-n} \right)\\ \notag
&=\sum_{n\in\mathbb Z}H^*_nz^{-n}.
\end{align}

 Note that * is $\mathbb Q(t)$-linear and  anti-involutive satisfying
\begin{equation}
\langle H_nu, v\rangle=\langle u, H_n^*v\rangle
\end{equation}
for $u, v\in \Lambda_F$.

We now introduce the \textit{vertex operators $Q(z)$} and {\it its adjoint $Q^*(z)$} as the linear maps: $\Lambda_F\longrightarrow \Lambda_F[[z, z^{-1}]]$ defined by
\begin{align}
\label{e:schurQop}
Q(z)&=\mbox{exp} \left( \sum\limits_{n\geq 1, \text{odd}} \dfrac{2}{n}p_nz^{n} \right) \mbox{exp} \left( -\sum \limits_{n\geq 1} \frac{\partial}{\partial p_n}z^{-n} \right)\\ \notag
&=\sum_{n\in\mathbb Z}Q_nz^{n},\\ \label{e:schurQop1}
Q^*(z)&=\mbox{exp} \left(-\sum\limits_{n\geq 1} \dfrac{1-t^{n}}{n}p_nz^{n} \right) \mbox{exp} \left(\sum \limits_{n\geq 1, \text{odd}} \frac{2}{1-t^n}\frac{\partial}{\partial p_n}z^{-n} \right)\\ \notag
&=\sum_{n\in\mathbb Z}Q^*_nz^{-n}.
\end{align}
The components $H_n, H_{-n}^*\in End_F(\Lambda)$ are of degree $n$, so are
annihilation operators for $n>0$. Similarly $Q_n, Q^*_{-n}\in End_{\mathbb Q}(\Lambda)$. We remark that the
second exponential factor of $Q(z)$ is different from the usual construction in \cite{Jing91}, and this will be crucial for
our later discussion. In particular, note that $Q(-z)\neq Q^*(z)$ in the current situation due to different inner product.

We collect the relations of the vertex operators as follows.
\begin{prop}\label{p:com}
\cite{Jing1, Jing91} 1) The operators $H_n$ and $H_n^*$ satisfy the following relations
\begin{align}\label{e:com1}
H_{m}H_n-tH_nH_m&=tH_{m+1}H_{n-1}-H_{n-1}H_{m+1}\\ \label{e:com2}
H^*_{m}H^*_n-tH^*_nH^*_m&=tH^*_{m-1}H^*_{n+1}-H^*_{n+1}H^*_{m-1}\\ \label{e:com3}
H_{m}H^*_n-tH^*_nH_m&=tH_{m-1}H^*_{n-1}-H^*_{n-1}H_{m-1}+(1-t)^2\delta_{m, n}\\ \label{e:com4}
H_{-n}. 1&=Q_{-n}.1=\delta_{n, 0}, \qquad H_{n}^{*}. 1=Q_{n}^{*}. 1=\delta_{n, 0}
\end{align}
where $\delta_{m, n}$ is the Kronecker delta function.

2) The operators $Q_n$ satisfy the Clifford algebra relations:
\begin{align}\label{e:com5}
\{Q_m, Q_n\}=(-1)^n2\delta_{m,-n}
\end{align}
where $\{A, B\}=AB+BA$.
\end{prop}
\begin{proof} Commutation relations \eqref{e:com1}-\eqref{e:com4} were from \cite{Jing1}. We focus on 2).
Define the normal ordering product
\begin{align*}
:Q(z)Q(w):=\mbox{exp} \left( \sum\limits_{n\geq 1, \text{odd}} \dfrac{2}{n}p_n(z^{n}+w^{n}) \right)\mbox{exp} \left( -\sum \limits_{n\geq 1} \frac{\partial}{\partial p_n}(z^{-n}+w^{-n}) \right).
\end{align*}
Then we have for $|z|<|w|$
\begin{align*}
Q(z)Q(w)&=:Q(z)Q(w):\mbox{exp} \left( -\sum\limits_{n\geq 1, \text{odd}} \dfrac{2}{n}(\frac{w}{z})^n \right)\\
&=:Q(z)Q(w):\frac{z-w}{z+w}.
\end{align*}
The rest of the argument is similar to Proposition 4.15 in \cite{Jing91}.
\end{proof}
Note that
the vacuum vector $1$ is annihilated by $p_n^*$, so
\begin{equation}\label{e:qfcn1}
H(z).1=\exp(\sum_{n=1}^{\infty}\frac{1-t^n}np_nz^n)=\sum_{n=0}^{\infty}q_nz^n=q(z)
\end{equation}
where $q_n$ is the Hall-Littlewood polynomial of one-row partition $(n)$, and clearly
\begin{equation}\label{e:qfcn2}
q_n=H_n.1=\sum_{\lambda\vdash n}\frac1{z_{\lambda}(t)}p_{\lambda}.
\end{equation}

We also introduce a spin analogue $h(z)$ as follows,
\begin{align}
\tilde{h}(z)=\exp(\sum_{n=1}^{\infty}\frac{t^n-(-1)^n}{n}p_nz^n)=\sum\limits_{n\geq0}\tilde h_nz^n,
\end{align}
then
\begin{align}
\tilde h_n=\sum\limits_{\lambda\vdash n}\frac{\varepsilon_{\lambda}}{z_{\lambda}(-t)}p_{\lambda}.
\end{align}
Moreover,
\begin{align}
\tilde h_{n}(-t)=\sum\limits_{\lambda\vdash n}\varepsilon_{\lambda}u_{\lambda}q_{\lambda},
\end{align}
where $\varepsilon_{\lambda}=(-1)^{|\lambda|-l(\lambda)}$ and
$u_{\lambda}=\frac{l(\lambda)!}{\prod_{i\geq1}m_{i}(\lambda)!}.$

As consequences of the proposition, one also has that 
\begin{align} \label{e:com5}
H_{n}H_{n+1}&=tH_{n+1}H_{n},\\ \label{e:com6}
H_{n}^{*}H_{n-1}^{*}&=tH_{n-1}^{*}H_{n}^{*},\\ \label{e:com7a}
\langle H_n.1, H_n.1\rangle&=\sum_{\lambda\vdash n}\frac{1}{z_{\la}(t)}=1-t, \qquad n>0\\ \label{e:com7b}
\langle H_n.1, H^*_{-n}.1\rangle&=\sum_{\lambda\vdash n}\frac{(-1)^{l(\lambda)}}{z_{\la}(t)}=t^n-t^{n-1}, \qquad n>0
\end{align}
where the last two identities follow from \eqref{e:com3} and \eqref{e:com1} by induction.

In general, expressing $H_{\mu}$ for any composition $\mu$ in terms of the basis elements $H_{\lambda}$, $\lambda\in\mathcal P$
can be formulated as follows.
Let $S_{i,a}$ be the transformation $(\la_1, \cdots, \la_i, \la_{i+1}, \cdots)\mapsto (\la_1, \cdots, \la_{i+1}-a, \la_{i}+a, \cdots)$, where
$\la_{i+1}>\la_i$.
Define
\begin{equation}\label{e:straight}
C(S_{i,a})=\begin{cases} t & a=0\\ t^{a+1}-t^{a-1} & 1\leq a< [\frac{\la_{i+1}-\la_i}2] \\ t^{a+\epsilon}-t^{a-1} & 1\leq a= [\frac{\la_{i+1}-\la_i}2]
\end{cases}
\end{equation}
where $\epsilon\equiv \la_{i+1}-\la_{i} (mod\, 2)$.
For $\underline{i}=(i_1, \ldots, i_r)$ and $\underline{a}=(a_1, \ldots, a_r)$ let
\begin{equation}
C(S_{\underline{i}, \underline{a}})=C(S_{i_1, a_1})C(S_{i_2, a_2})\cdots C(S_{i_r, a_r})
\end{equation}
where the product order follows that of $S_{i_1, a_1}S_{i_2, a_2}\cdots S_{i_r, a_r}\la$, i.e. from the right to the left.
In particular, when $t=0$, $C(S_{\underline{i}, \underline{a}})=0$ unless all $a_i=1$, then $C(S_{\underline{i}, \underline{1}})=(-1)^r$
which is possible only when $\la_{i+1}-\la_i\geq 2$.
When $t=-1$, $C(S_{\underline{i}, \underline{a}})=0$ unless when all $a_i=0$ and $C(S_{\underline{i}, \underline{0}})=(-1)^r$.

Let $\mu$ be a composition and $\lambda$ be a partition. Denote
\begin{align}\label{e:def}
B(\lambda, \mu)\triangleq \sum\limits_{\underline{i},\underline{a}}C(S_{\underline{i},\underline{a}})
\end{align}
summed over $\underline{i}=(i_1,i_2,\ldots,i_r)$, $\underline{a}=(a_1,a_2,\ldots,a_r)$ such that $S_{\underline{i},\underline{a}}\mu=\lambda.$

\begin{prop} \cite{Jing2} \label{p:straight} Suppose $\mu$ is a composition, then
\begin{equation}\label{e:straight2}
H_{\mu}=\sum\limits_{\lambda\vdash |\mu|} B(\lambda,\mu) H_{\la}.
\end{equation}
\end{prop}

We remark that $\lambda$ appears only when $\lambda\geq\mu$ in \eqref{e:straight2}. Let $\mu$ be a composition and $\lambda$ be a partition. If there exists $\underline{i}=(i_1,i_2,\ldots,i_r)$, $\underline{a}=(a_1,a_2,\ldots,a_r)$ such that $S_{\underline{i},\underline{a}}\mu=\lambda.$ Then $\sum\limits_{i=1}^k\lambda_i\geq\sum\limits_{i=1}^k\mu_i,$ $k=1,2,\cdots$

\begin{prop} \cite{Jing1, Jing91} \label{t:HL} (1) Let $\lambda=(\lambda_{1},\ldots ,\lambda_{l})$ be a partition.
The vertex operator products  $H_{\lambda_{1}}\cdots H_{\lambda_{l}}. 1$ is the
Hall-Littlewood function $Q_{\la}(t)$:
\begin{equation}\label{e:HL}
H_{\lambda_{1}}\cdots H_{\lambda_{l}}. 1=Q_{\la}(t)=
\prod\limits_{i<j} \dfrac{1-R_{ij}}{1-tR_{ij}}q_{\lambda_{1}}\cdots q_{\lambda_{l}}
\end{equation}
where the raising operator $R_{ij}q_{\la}=q_{(\la_{1},\ldots ,\la_{i}+1,\ldots ,\la_{j}-1,\ldots , \la_{l})}$.

(2) Let $\xi=(\xi_1,\xi_2,\ldots,\xi_l)$ be a strict partition, then
\begin{align}
Q_{\xi}=Q_{\xi_1}Q_{\xi_2}\cdots Q_{\xi_l}.1
\end{align}
are the Schur Q-function indexed by the strict partition $\xi$.
Moreover, $Q_{\xi}.1$, $\xi$ strict, form an orthogonal $\mathbb{Z}$-base of $\Gamma$ under the specialized
inner product $\langle \ , \ \rangle_{t=-1}$, explicitly
\begin{align}\label{e:orth}
\langle Q_{\lambda}.1, Q_{\xi}.1\rangle|_{t=-1}=2^{l(\lambda)}\delta_{\lambda\xi},  \qquad \lambda, \xi\in\mathcal{SP}.
\end{align}
\end{prop}
\begin{proof} Part (1) is from \cite{Jing1}. Since our vertex operator $Q(z)$ is different from that of \cite{Jing91}, we explain why the new vertex operator
also realizes the Schur Q-functions. From the argument in proving \eqref{e:com5} in Prop. \ref{p:com} it follows that
\begin{align*}
Q(z_1)Q(z_2)\cdots Q(z_l).1&=\prod_{i<j}\frac{z_i-z_j}{z_i+z_j}:Q(z_1)Q(z_2)\cdots Q(z_l):.1\\
&=\prod_{i<j}\frac{z_i-z_j}{z_i+z_j}\exp(\sum_{n\geq 1, \mathrm{odd}}\frac{2p_n}{n}(z_1^n+\cdots n_l^n)).
\end{align*}
Taking coefficient of $z_1^{\xi_1}\cdots z_l^{\xi_l}$, we obtain that  $Q_{\xi}$ is exactly the Schur Q-function indexed by $\xi$ (cf. \cite{Jing91}).
\end{proof}

\section{Spin Hall-Littlewood polynomials and vertex operators}\label{S:spinHL}
Wan and Wang have introduced a very interesting spin analogue of Kostka(-Foulkes) polynomials and shown that these polynomials enjoy favorable properties parallel to those of the Kostka polynomials \cite{W}.

\begin{defn} \cite{W} \label{t:spin K-F}
The spin Kostka polynomials $K^-_{\xi\mu}(t)$ for $\xi\in\mathcal{SP}$ and $\mu\in\mathcal{P}$ are defined by
\begin{align}\label{e:spin K-F}
Q_{\xi}(x)=\sum\limits_{\mu}K^-_{\xi\mu}(t)P_{\mu}(x;t),
\end{align}
where $Q_{\xi}(x)$ (resp.$P_{\mu}(x;t)$) are Schur's Q-functions (resp. Hall-Littlewood functions).
\end{defn}
From the above discussion and Theorem \ref{t:HL}, it is clear that the spin Kostka polynomials can be expressed as the matrix coefficients:
\begin{align*}
K^-_{\xi\mu}(t)&=\langle Q_{\mu}(x;t), Q_{\xi}(x) \rangle\\
&=\langle H_{\mu_1}H_{\mu_2}\cdots H_{\mu_l}.1, Q_{\xi_1}Q_{\xi_2}\cdots Q_{\xi_k}.1 \rangle.
\end{align*}

To compute the matrix coefficients, we first get the commutation relations by usual techniques of vertex operators:
\begin{align}
\label{e:hallop1}
H^*(z)Q(w)(w-tz)+Q&(w)H^*(z)(z+w)=2(1-t)z\delta(\frac{w}{z})\tilde h(z),
\\ \label{e:hallop2}
\tilde h^{*}(z)H(w)&=H(w)\tilde h^{*}(z)\frac{w+z}{w-tz},\\\label{e:hallop2}
Q(z)\tilde h(w)&=\tilde h(w)Q(z)\frac{z-tw}{z+w}.
\end{align}

We remark that if the old vertex operator $\tilde{Q}(w)$ from \cite{Jing91} were used, then the commutation relations between $H^*(z)$ and $\tilde{Q}(w)$ would have been an infinite quadratic relation.

Taking coefficients we obtain the following commutation relations.
\begin{prop}\label{p:rel1}
The commutation relations between the Hall-Littlewood vertex operators and Schur's $Q$-function operators are:
\begin{align}\label{e:rel1}
H^{*}_{n}Q_{m}&=t^{-1}H^{*}_{n-1}Q_{m-1}+t^{-1}Q_{m}H^{*}_{n}+t^{-1}Q_{m-1}H^*_{n-1}\\ \nonumber
&\qquad+2(1-t^{-1})\tilde h_{m-n},\\ \label{e:rel2}
\tilde h^{*}_{m}H_{n}&=H_{n}\tilde h^{*}_{m}+(1+t)\sum\limits_{k=0}^{m-1}t^{m-k-1}H_{n-m+k}\tilde h^*_{k},\\\label{e:rel3}
Q_n\tilde h_m&=\tilde h_mQ_n+(1+t)\sum\limits_{k=0}^{m-1}(-1)^{m-k}\tilde h_kQ_{n-k+m}.
\end{align}
\end{prop}

Now we can state our formulas to compute the spin Kostka polynomials. To this end, we prepare some necessary notations.
Let $\lambda=(\lambda_1,\lambda_2,\ldots,\lambda_l)$ and $\mu=(\mu_1,\mu_2,\ldots,\mu_m)$ be (strict)partitions.
We denote $\lambda^{[i]}=(\lambda_{i+1},\ldots,\lambda_l)$, $\lambda^{\hat{i}}=(\lambda_1,\ldots,\lambda_{i-1},\lambda_{i+1},\ldots,\lambda_l)$, and $\lambda-\mu=(\lambda_1-\mu_1,\lambda_2-\mu_2,\cdots).$
\begin{thm}\label{t:iterative}
For strict partition $\xi=(\xi_1,\xi_2,\ldots,\xi_l)$ and partition $\mu=(\mu_1,\mu_2,\ldots,\mu_m)$ and integer $k$,
\begin{align}\label{e:iterative}
H_{k}^{*}Q_{\xi}&=\sum\limits_{i=1}^{l} (-1)^{i-1}2\tilde h_{\xi_i-k}Q_{\xi^{\hat{i}}}, 
\\\label{e:hH}
\tilde h^{*}_{k}H_{\mu}&=\sum\limits_{\tau\models k}t^{k-l(\tau)}(1+t)^{l(\tau)}H_{\mu-\tau}.
\end{align}
\end{thm}
\begin{proof} We show the first relation by induction on $k+|\xi|$.
The case of $k+|\xi|=1$ is clear. Assume that \eqref{e:iterative} holds for $k+|\xi|=n-1$. Using the induction hypothesis
and \eqref{e:rel1} we have that
\begin{align*}
&H^*_{k}Q_{\xi_1}Q_{\xi_2}\cdots Q_{\xi_l}\\
=&t^{-1}H^*_{k-1}Q_{\xi_1-1}Q_{\xi_2}\cdots Q_{\xi_l}+t^{-1}Q_{\xi_1}H^*_{k}Q_{\xi_2}\cdots Q_{\xi_l}+t^{-1}Q_{\xi_1-1}H^*_{k-1}Q_{\xi_2}\cdots Q_{\xi_l}\\
+&2(1-t^{-1})\tilde h_{\xi_{1}-k}Q_{\xi_2}\cdots Q_{\xi_l}\\
=&t^{-1}(2\tilde h_{\xi_1-k}Q_{\xi_2}\cdots Q_{\xi_l}-2\tilde h_{\xi_2-k+1}Q_{\xi_1-1}Q_{\xi_3}\cdots Q_{\xi_l}+2\tilde h_{\xi_3-k+1}Q_{\xi_1-1}Q_{\xi_2}Q_{\xi_4}\cdots Q_{\xi_l}\\
+&\cdots+(-1)^{l+1}2\tilde h_{\xi_l-k+1}Q_{\xi_1-1}Q_{\xi_2}\cdots Q_{\xi_{l-1}})\\
+&t^{-1}Q_{\xi_1}(2\tilde h_{\xi_2-k}Q_{\xi_3}\cdots Q_{\xi_l}-2\tilde h_{\xi_3-k}Q_{\xi_2}Q_{\xi_4}\cdots Q_{\xi_l}+2\tilde h_{\xi_4-k}Q_{\xi_2}Q_{\xi_3}Q_{\xi_5}\cdots Q_{\xi_l}\\
+&\cdots+(-1)^{l}2\tilde h_{\xi_l-k}Q_{\xi_2}Q_{\xi_3}\cdots Q_{\xi_{l-1}})\\
+&t^{-1}Q_{\xi_1-1}(2\tilde h_{\xi_2-k+1}Q_{\xi_3}\cdots Q_{\xi_l}-2\tilde h_{\xi_3-k+1}Q_{\xi_2}Q_{\xi_4}\cdots Q_{\xi_l}+2\tilde h_{\xi_4-k+1}Q_{\xi_2}Q_{\xi_3}Q_{\xi_5}\cdots Q_{\xi_l}\\
+&\cdots+(-1)^{l}2\tilde h_{\xi_l-k+1}Q_{\xi_2}Q_{\xi_3}\cdots Q_{\xi_{l-1}})\\
+&2(1-t^{-1})\tilde h_{\xi_1-k}Q_{\xi_2}Q_{\xi_3}\cdots Q_{\xi_l}.
\end{align*}
Simplifying the expression, we see the above is 
\begin{align*}
&t^{-1}(2\tilde h_{\xi_1-k}Q_{\xi_2}\cdots Q_{\xi_l}-2\tilde h_{\xi_2-k+1}Q_{\xi_1-1}Q_{\xi_3}\cdots Q_{\xi_l}+2\tilde h_{\xi_3-k+1}Q_{\xi_1-1}Q_{\xi_2}Q_{\xi_4}\cdots Q_{\xi_l}\\
+&\cdots+(-1)^{l+1}2\tilde h_{\xi_l-k+1}Q_{\xi_1-1}Q_{\xi_2}\cdots Q_{\xi_{l-1}})\\
+&2t^{-1}(\tilde h_{\xi_2-k+1}Q_{\xi_1-1}Q_{\xi_3}\cdots Q_{\xi_l}-\tilde h_{\xi_3-k+1}Q_{\xi_1-1}Q_{\xi_2}Q_{\xi_4}\cdots Q_{\xi_l}+\tilde h_{\xi_4-k+1}Q_{\xi_1-1}\cdots Q_{\xi_l}\\
+&\cdots+(-1)^{l}\tilde h_{\xi_l-k+1}Q_{\xi_1-1}Q_{\xi_2}Q_{\xi_3}\cdots Q_{\xi_{l-1}})\\
-&2(\tilde h_{\xi_2-k}Q_{\xi_1}Q_{\xi_3}\cdots Q_{\xi_l}-\tilde h_{\xi_3-k}Q_{\xi_1}Q_{\xi_2}Q_{\xi_4}\cdots Q_{\xi_l}+\tilde h_{\xi_4-k}Q_{\xi_1}Q_{\xi_2}Q_{\xi_3}Q_{\xi_5}\cdots Q_{\xi_l}\\
+&\cdots+(-1)^{l}\tilde h_{\xi_l-k}Q_{\xi_1}Q_{\xi_2}Q_{\xi_3}\cdots Q_{\xi_{l-1}})\\
+&2(1-t^{-1})\tilde h_{\xi_1-k}Q_{\xi_2}Q_{\xi_3}\cdots Q_{\xi_l} \quad(\text{by \eqref{e:rel3}})\\
=&2\tilde h_{\xi_1-k}Q_{\xi_2}\cdots Q_{\xi_l}-2\tilde h_{\xi_2-k}Q_{\xi_1}Q_{\xi_3}\cdots Q_{\xi_l}+\cdots+2(-1)^{l-1}\tilde h_{\xi_l-k}Q_{\xi_1}\cdots Q_{\xi_{l-1}},
\end{align*}
which has proved the first one. The second relation is similarly shown by \eqref{e:rel2} and induction on $l(\mu)$ as well.
\end{proof}

\begin{exmp}
Let $\mu=(2,2)$ and $\xi=(3,1)$, then by Theorem \ref{t:iterative}
\begin{align*}
K^-_{\xi\mu}(t)&=\langle H_2H_2.1, Q_3Q_1.1 \rangle\\
&=\langle H_2.1, 2h_1Q_1.1 \rangle\\
&=2\langle t(1+t^{-1})H_1.1, Q_1.1 \rangle\\
&=4t+4.
\end{align*}
\end{exmp}

By Theorem \ref{t:iterative}, we now obtain an algebraic formula for $K^-_{\xi\mu}(t).$
\begin{thm}\label{t:Recurrence Formula}
For $\xi=(\xi_1,\cdots,\xi_{l})\in\mathcal{SP}_n$ and $\mu=(\mu_1,\cdots,\mu_m)\in\mathcal{P}_n$, $K^-_{\xi\mu}(t)$ is given by the iterative formula:
\begin{align}\label{e:Recurrence Formula}
&K^-_{\xi\mu}(t)\\ \nonumber
=&\sum\limits_{i=1}^{l}\sum\limits_{\tau\models\xi_i-\mu_1}\sum\limits_{\lambda\vdash n-\xi_i}(-1)^{i-1}2t^{\xi_i-\mu_1}(1+t^{-1})^{l(\tau)}B(\lambda,\mu^{[1]}-\tau)K^-_{\xi^{\hat{i}}\lambda}(t).
\end{align}
\end{thm}
\begin{proof}
It follows readily from \eqref{e:iterative}, \eqref{e:hH} and \eqref{e:straight2}.
\end{proof}

Eq. \eqref{e:Recurrence Formula} shows that all spin Kostka polynomials are integral polynomials, and it also gives an effective recurrence of $K^-_{\xi\mu}(t)$ as shown by the
following example. 
\begin{exmp} Let $\xi=(4,3,1)$ and $\mu=(3,3,2)$, then
\begin{align*}
K^-_{\xi\mu}(t)&=\langle H_3H_3H_2.1, Q_4Q_3Q_1.1 \rangle\\
&=\langle H_3H_2.1, 2\tilde h_1Q_3Q_1.1 \rangle-\langle H_3H_2.1, 2\tilde h_0Q_4Q_1.1 \rangle\\
&=2\langle t(1+t^{-1})(H_2H_2.1+H_3H_1.1), Q_3Q_1.1 \rangle-2\langle H_3H_2.1, Q_4Q_1.1 \rangle\\
&=2(t+1)(K^-_{(3,1)(2,2)}(t)+K^-_{(3,1)(3,1)}(t))-2K^-_{(4,1)(3,2)}(t).
\end{align*}
\end{exmp}

The spin Kostka polynomials have quite a few remarkable properties resembling those of the Kostka-Foulkes polynomials. As a consequence of the recurrence we have the following.
\begin{cor}
Let $\xi$ be a strict partition and $\mu$ be a partition, we have\\
\indent $(1)$ If there exists $k\in\mathbb{N},$ such that $\xi_i=\mu_i, i=1,2,\cdots,k$, then
\begin{align}
K^-_{\xi\mu}(t)=2^kK^-_{\xi^{[k]}\mu^{[k]}}(t).
\end{align}
In particular, $K^-_{\xi\xi}(t)=2^{l(\xi)}.$\\
\indent $(2)$ $2^{l(\xi)}\mid K^-_{\xi\mu}(t).$\\
\indent $(3)$ $K^-_{\xi\mu}(-1)=2^{l(\xi)}\delta_{\xi\mu}.$
\end{cor}
\begin{proof}
They are immediate consequences of Theorem \ref{t:Recurrence Formula}.
\end{proof}

Some special cases of Theorem \ref{t:Recurrence Formula} are listed as follows. 
\begin{exmp}Suppose $\xi\in\mathcal{SP}_n,$ $\mu\in\mathcal{P}_n,$ we have
\begin{align}
K^-_{\xi(n)}(t)&=2\delta_{\xi,(n)}\\
K^-_{(n)\mu}(t)&=2t^{n-\mu_1}\sum\limits_{\tau\models n-\mu_1}(1+t^{-1})^{l(\tau)}B(\emptyset, \mu^{(1)}-\tau)\\
K^-_{\xi(\mu_1,\mu_2)}(t)&=
\begin{cases}
2^{2-\delta_{0,\xi_2}}t^{\xi_1-\mu_1}(1+t^{-1})& \text{if $\xi>(\mu_1,\mu_2)$}\\
4& \text{if $\xi=(\mu_1,\mu_2)$}\\
0& \text{if others.}
\end{cases}
\end{align}
\end{exmp}

There is a compact formula of $K^-_{(n)\mu}(t)$ \cite{W} by using a result of \cite{M}. We will
come back to the Wan-Wang formula using the iteration in the next section.

The following result was first proved in \cite{W} using the similar property of the Kostka-Foulkes polynomials. Using our iterative
formula, one can give an independent proof from that of the Kostka-Foulkes polynomials. We remark that the method can also be used to show
this property for the Kostka-Foulkes polynomial by the iterative formula in \cite{BJ}.
\begin{cor}
Let $\xi=(\xi_1,\xi_2,\cdots)\in\mathcal{SP}_n,$ $\mu=(\mu_1,\mu_2,\cdots)\in\mathcal{P}_n,$ then $K^-_{\xi\mu}(t)=0,$ unless $\xi\geq\mu.$
\end{cor}
\begin{proof}
It is equivalent to prove $K^-_{\xi\mu}(t)=0,$ if $\xi\ngeq\mu.$ We argue it by induction on $n.$ The initial step is obvious. Suppose it holds for $<n.$ There exists a smallest $k\geq1$, such that $\xi_1+\xi_2+\cdots+\xi_k<\mu_1+\mu_2+\cdots+\mu_k.$ \\
If $k=1$, then it's evident that $K^-_{\xi\mu}(t)=0$ by the iterative formula \eqref{e:Recurrence Formula}.\\
If $k>1$, then there exists $k>j\geq1,$ such that $\xi_{j+1}<\mu_1\leq\xi_j.$ We have
\begin{align*}
K^-_{\xi\mu}(t)&=\sum\limits_{i=1}^{j}(-1)^{i-1}\langle H_{\mu_2}H_{\mu_3}\cdots, 2\tilde h_{\xi_i-\mu_1}Q_{\xi_1}\cdots\hat{Q}_{\xi_i}\cdots \rangle\\
&=\sum\limits_{i=1}^{j}(-1)^{i-1}\sum\limits_{\tau\models\xi_i-\mu_1}2t^{\xi_i-\mu_1}(1+t^{-1})^{l(\tau)}\langle H_{\mu^{[1]}-\tau}, Q_{\xi^{\hat{i}}} \rangle\\
&=\sum\limits_{i=1}^{j}(-1)^{i-1}\sum\limits_{\tau\models\xi_i-\mu_1}2t^{\xi_i-\mu_1}(1+t^{-1})^{l(\tau)}\sum\limits_{\nu\vdash n-\xi_i}
B(\nu,\mu^{[1]}-\tau)\langle H_{\nu}, Q_{\xi^{\hat{i}}} \rangle\\
&=\sum\limits_{i=1}^{j}(-1)^{i-1}\sum\limits_{\tau\models\xi_i-\mu_1}2t^{\xi_i-\mu_1}(1+t^{-1})^{l(\tau)}\sum\limits_{\nu\vdash n-\xi_i}
B(\nu,\mu^{[1]}-\tau)K^-_{\xi^{\hat{i}}\nu}(t).
\end{align*}
By the remark below Proposition \ref{p:straight}, for each $1\leq i\leq j$, we have $\nu_1+\cdots+\nu_{k-1}\geq\mu_2+\cdots+\mu_k-\tau_1-\cdots-\tau_{k-1}\geq\mu_2+\cdots+\mu_k+\mu_1-\xi_i>\xi_1+\cdots+\xi_{i-1}+\xi_{i+1}+\cdots\xi_k.$
By induction, we have $K^-_{\xi\mu}(t)=0.$
\end{proof}

The Kostka-Foulkes polynomials have the stability property \cite{BJ}, which says that if $\mu_1\geq\lambda_2$, then $K_{\lambda+(r),\mu+(r)}(t)=K_{\lambda\mu}(t)$ for all $r\geq1.$ Here, $\lambda+(r)=(\lambda_1+r,\lambda_2,\cdots).$ The spin Kostka polynomials also enjoy the same stability.
\begin{prop}
Let $\xi=(\xi_1,\cdots,\xi_l)\in\mathcal{SP},$ $\mu=(\mu_1,\cdots,\mu_m)\in\mathcal{P},$ and $\mu_1>\xi_2$. Then for any $r\geq1,$ we have
\begin{align}
K^-_{\xi+(r)\mu+(r)}(t)=K^-_{\xi\mu}(t).
\end{align}
\end{prop}
\begin{proof}
By Theorem \ref{t:iterative}, it follows that
\begin{align*}
K^-_{\xi+(r)\mu+(r)}(t)=\langle H_{\mu_2}H_{\mu_3}\cdots H_{\mu_m}.1, 2\tilde h_{\xi_1-\mu_1}Q_{\xi_2}\cdots Q_{\xi_l}.1 \rangle=K^-_{\xi\mu}(t).
\end{align*}
\end{proof}

The spin Kostka-Foulkes polynomials $K_{\lambda\mu}(t)$ were conjecturally symmetric \cite[Qu. 4.10]{W} in the sense that
\begin{align*}
K^-_{\lambda\mu}(t)=t^{m_{\lambda\mu}}K^-_{\lambda\mu}(t^{-1})
\end{align*}
for some $m_{\lambda\mu}\in\mathbb{Z}$. 
However, the following is a counterexample.
\begin{exmp} \label{e:counter} Given $\xi=(3,2)$ and $\mu=(2,1^3),$ we have
\begin{align*}
K^-_{\xi\mu}(t)&=\langle H_2H_1H_1H_1.1, Q_3Q_2.1 \rangle\\
&=\langle H_1H_1H_1.1, 2\tilde h_1Q_2.1 \rangle-\langle H_1H_1H_1.1, 2\tilde h_0Q_3.1 \rangle\\
&=2\langle t(1+t^{-1})[3]H_1H_1.1, Q_2.1 \rangle-2K^-_{(3)(1^3)}(t)\\
&=4t(t^3+2t^2+3t+2).
\end{align*}
\end{exmp}

\section{Marked Tableaux}\label{S:MT}
To study projective representations of the symmetric group, Stembridge \cite{S} introduced the number $g_{\xi\lambda}$ as follows:
\begin{align}\label{e:bdefn}
Q_{\xi}(x)=\sum\limits_{\lambda}b_{\xi\lambda}s_{\lambda}(x), \quad g_{\xi\lambda}=2^{-l(\xi)}b_{\xi\lambda}.
\end{align}
Note that $b_{\xi\lambda}=K^-_{\xi\lambda}(0)$, but we will see that $g_{\xi\lambda}$ can be extended to any partition $\xi$, so we reserve this notation in this section.

Let $\xi,\lambda$ be partitions with $\xi$ strict. The coefficient $g_{\xi\lambda}$ of $s_{\lambda}$ in the expansion of the Schur Q-function
$2^{-l(\xi)}Q_{\xi}$ counts the number of (unshifted) marked tableaux $T$ of shape $\lambda$ and weight $\xi$ such that\\
\indent (a) $w(T)$ has the lattice property;\\
\indent (b) for each $k\geq1,$ the last occurrence of $k^{'}$ in $w(T)$ precedes the last occurrence of $k.$\\
Here $w(T)$ is the word of $T$ by reading the symbols in $T$ from right to left in successive rows, starting with the top row.

The combinatorial interpretation and the representation-theoretic interpretation of $g_{\xi\lambda}$ are known \cite{Sa, S, W, Wo}. However, no effective formula for $g_{\xi\mu}$ is available. As an application of the preceding section, we give an algebraic formula for $g_{\xi\lambda}$.

The ring $\Lambda_\mathbb{Q}$ of symmetric functions has the canonical
bilinear form $\langle\ , \ \rangle_0=\langle\ , \ \rangle_{t=0}$ under which Schur functions are orthonomal:
\begin{align}
\langle p_{\lambda}, p_{\mu}\rangle_0=\delta_{\lambda,\mu}z_{\lambda},
\end{align}
thus the adjoint operator of the multiplication operator $p_n$
is the differential operator $p_n^- =n\frac{\partial}{\partial p_n}$.

With respect to $\langle\ , \ \rangle_0$, the {\em vertex operators} and their adjoint operators for Schur's functions and Schur's $Q$-functions are given by \cite{Jing91, Jing00}:
\begin{align}
S^{\pm}(z)&=\mbox{exp} \left( \pm\sum\limits_{n\geq 1} \dfrac{1}{n}p_nz^{n} \right) \mbox{exp} \left( \mp\sum \limits_{n\geq 1} \frac{\partial}{\partial p_n}z^{-n} \right)\\ \notag
&=\sum_{n\in\mathbb Z}S^{\pm}_nz^{\pm n},\\
Q^{+}(z)&=Q(z)=\sum_{n\in\mathbb Z}Q^{+}_nz^{ n},\\
Q^-(z)&=\mbox{exp} \left(-\sum\limits_{n\geq 1} \dfrac{1}{n}p_nz^{n} \right) \mbox{exp} \left(\sum \limits_{n\geq 1, \text{odd}} 2\frac{\partial}{\partial p_n}z^{-n} \right)\\ \notag
&=\sum_{n\in\mathbb Z}Q^-_nz^{-n}.
\end{align}
Note that $Q^{-}(z)$ are the specialized vertex operator 
$Q^*(z)|_{t=0}$. Here we denote the adjoint operators by $S^+_n$ and $Q^+_n$ respectively, to distinguish from the
preceding section.

Therefore $g_{\xi\lambda}$ can be expressed in terms of this inner product:
\begin{align}
g_{\xi\lambda}=2^{-l(\xi)}b_{\xi\lambda}=2^{-l(\xi)}\langle s_{\lambda}, Q_{\xi} \rangle_0=2^{-l(\xi)}\langle S_{\lambda}.1, Q_{\xi}.1 \rangle_0.
\end{align}

Recall that the involution $\omega: \Lambda \rightarrow\Lambda$ defined by $\omega(p_{\lambda})=\varepsilon_{\lambda}p_{\lambda}$ \cite{M} is an isometry with respect to
the canonical inner product $\langle\ , \ \rangle_0$ such that
\begin{align*}
\omega(s_{\lambda})=s_{\lambda^{'}}, \quad \omega(Q_{\xi})=Q_{\xi}.
\end{align*}
\begin{prop} For given $\lambda\in\mathcal{P}_n,$ $\xi\in\mathcal{SP}_n,$ then $g_{\xi\lambda}$ or $ b_{\xi\lambda}$ has the following property
\begin{align}
g_{\xi\lambda}=g_{\xi\lambda^{'}}.
\end{align}
\end{prop}

We introduce the operators for the elementary symmetric functions $e_n$: 
\begin{align}
e^{\pm}(z)=\mbox{exp} \left(\sum\limits_{n\geq 1} \dfrac{(-1)^{n+1}}{n}p^{\pm}_nz^{\pm n} \right)=\sum\limits_{n\geq0}e^{\pm}_nz^{\pm n}.
\end{align}
where $p^+_n=p_n$, $p^-_n=n\frac{\partial}{\partial p_n}$, and $e^+(z)=h(z)|_{t=0}$.

Then by Theorem \ref{t:iterative} we have
\begin{prop}\label{t:iterative2}
For any strict partition $\xi=(\xi_1,\xi_2,\ldots,\xi_l)$, any partition $\lambda=(\lambda_1,\lambda_2,\ldots)$ and integer $k$,
\begin{align}\label{e:iterative2}
S_{k}^{-}Q_{\xi}&=\sum\limits_{i=1}^{l} (-1)^{i-1}2e_{\xi_i-k}Q_{\xi_1}Q_{\xi_2}\cdots \hat{Q}_{\xi_i}\cdots Q_{\xi_l},\\\label{e:gS}
e^{-}_{k}S_{\lambda}&=\sum\limits_{\rho}S_{\rho},
\end{align}
where $\rho$ runs through the partitions such that $\lambda/\rho$ are vertical $k$-strips.
\end{prop}

The algebraic iterative formula for $b_{\xi\lambda}$ is then natural:
\begin{thm}\label{t:Recurrence Formula2} Let $\xi\in\mathcal{SP}_n,$ $\lambda\in\mathcal{P}_n,$
\begin{align}
b_{\xi\lambda}=\sum\limits_{i=1}^{l(\xi)}2(-1)^{i-1}\sum\limits_{\rho^i}b_{\xi^{(i)}\rho^{i}},
\end{align}
where $\rho^i$ runs through the partitions such that $\lambda^{[1]}/\rho^i$ are vertical $\xi_i-\lambda_1$-strips.
\end{thm}

\begin{exmp} Let $\lambda\in\mathcal{P}_n,$ we have
\begin{align}\label{e:b(n)}
b_{(n)\lambda}=
\begin{cases}
2& \text{if $\lambda$ is a hook};\\
0& \text{if $\lambda$ is not a hook.}
\end{cases}
\end{align}
\end{exmp}

Combining \eqref{e:spin K-F} and \eqref{e:bdefn}, we have
\begin{align}
K^-_{\xi\mu}(t)=\sum\limits_{\lambda}b_{\xi\lambda}K_{\lambda\mu}(t),
\end{align}
where $K_{\lambda\mu}(t)$ are the Kostka-Foulkes polynomials.\\

By \eqref{e:b(n)}, we have
\begin{align}
K^-_{(n)\mu}(t)=\sum\limits_{\lambda \text{ hook}}2K_{\lambda\mu}(t).
\end{align}

Recall that a compact formula for the Kostka-Foulkes polynomials $K_{\lambda\mu}(t)$ is known for $\lambda$ being hook-shaped \cite{K, BJ}
\begin{align}
K_{(n-k,1^k)\mu}(t)=t^{n(\mu)+\frac{k(k+1-2l)}{2}}\left[\begin{matrix}l-1\\k\end{matrix}\right],
\end{align}
where $n=|\mu|, l=l(\mu).$ Therefore, we have that for any partition $\mu\vdash n$
\begin{align}\label{e:K(n)}
K^-_{(n)\mu}(t)&=\sum\limits_{k=0}^{l(\mu)-1}2t^{n(\mu)+\frac{k(k+1-2l(\mu))}{2}}\left[\begin{matrix}l(\mu)-1\\k\end{matrix}\right]\\ \label{e:K(n)1}
&=t^{n(\mu)}\prod\limits_{i=1}^{l(\mu)}(1+t^{1-i}).
\end{align}
Here the second equation follows from the $t$-binomial expansion \cite[(2.9)]{An} or an easy induction on $l(\mu)$ from \eqref{e:K(n)}. We remark that \eqref{e:K(n)1}
was first given by Wan-Wang \cite{W} using identities of Hall-Littlewood polynomials.


For given partition $\lambda,$ we define
\begin{align*}
\{\lambda\}_s\doteq\{\rho\subset\lambda^{[1]}\mid \rho~~ \text {is a hook and} ~~\lambda^{[1]}/\rho~~ \text {is a vertical} ~~s\text{-strips}\}.
\end{align*}
Set $N^{(s)}(\lambda)=\text{Card}\{\lambda\}_s.$ It is clear that $N^{(s)}(\lambda)=0$ when $s<0$ or $s>|\lambda^{[1]}|.$ Now we can give a two-row formula by the iterative formula for $b_{\xi\lambda}.$
\begin{thm}\label{t:tworow}
Let $1\leq m<\frac{n}{2},$ $\lambda\in\mathcal{P}_n,$ we have
\begin{align}\label{e:tworow}
b_{(n-m,m)\lambda}=4(N^{(n-m-\lambda_1)}(\lambda)-N^{(m-\lambda_1)}(\lambda)).
\end{align}
\end{thm}

To compute $N^{(s)}(\lambda),$ we denote all hook (resp. double hook) partitions of $n$ by $HP(n)$ (resp. $DHP(n)$). That is $HP(n)\doteq\{(\lambda_1,1^{m_1})\mid \lambda_1+m_1=n\},$ $DHP(n)\doteq\{(\lambda_1,\lambda_2,2^{m_2},1^{m_1})\mid \lambda_1+\lambda_2+2m_2+m_1=n\}.$ Clearly, $HP(n)\subset DHP(n).$ We remark that $N^{(s)}(\lambda)=0$ unless $\lambda\in DHP(n).$ Now let's consider $N^{(s)}(\lambda)~~(0\leq s\leq|\lambda^{[1]}|,~\lambda\in DHP(n))$ case by case.

{\bf Case 1:} If $\lambda\in HP(n),$ then $N^{(s)}(\lambda)=1$.\\
Before considering the case $\lambda\in DHP(n)\backslash HP(n),$ we look at the following special case.

{\bf Case 2:} If $\lambda=(\lambda_1,\lambda_2,1^{m_1})$ and $\lambda\notin HP(n),$ then we have
\begin{align}
N^{(s)}(\lambda)=
\begin{cases}
0& \text{if $s\geq m_1+2$}\\
1&\text{if $s=0$ or $s=m_1+1$}\\
2&\text{if $1\leq s\leq m_1$}.
\end{cases}
\end{align}

{\bf Case 3:} If $\lambda=(\lambda_1,\lambda_2,2^{m_2},1^{m_1})\in DHP(n)\backslash HP(n)$. It follows from {\bf Case 2} that
\begin{align}
\begin{split}
N^{(s)}(\lambda)&=N^{(s-m_2)}((\lambda_1,\lambda_2,1^{m_1}))\\&=
\begin{cases}
0& \text{if $0\leq s\leq m_2-1$ or $s\geq m_1+m_2+2$}\\
1&\text{if $s=m_2$ or $s=m_1+m_2+1$}\\
2&\text{if $1+m_2\leq s\leq m_1+m_2$}.
\end{cases}
\end{split}
\end{align}

\begin{exmp}
For given $\xi=(4,3),$ $\lambda=(2,2,2,1),$ then $\lambda_1=\lambda_2=2,$ $m_1=m_2=1,$ we have
\begin{align*}
b_{(4,3)(2,2,2,1)}=4(N^{(2)}(\lambda)-N^{(1)}(\lambda))=4\times(2-1)=4.
\end{align*}
\end{exmp}

The symmetric group $\mathfrak{S}_n$ has a two-valued representation, known as the spin representation studied by Schur and this is actually a representation of the double covering group
$\widetilde{\mathfrak{S}}_n$ of $\mathfrak S_n$ \cite{Sc}.  It is known that the irreducible spin representations of $\mathfrak S_n$ are parametrized by strict partitions of $n$.
Let $\zeta^{\lambda}$ be the irreducible spin character of the Schur double covering group $\tilde{\mathfrak{S}}_n$ afforded by the module $V^{\lambda}$, $\lambda\in\mathcal{SP}_n$. Stembridge \cite{S} obtained the irreducible decomposition for the twisted tensor product of $\widetilde{\mathfrak{S}}_n$ \cite{Kl}:
\begin{align*}
\text{\bf ch}(\zeta^{(n)}\otimes \zeta^{\lambda})=P_{\lambda}(x;-1)
\end{align*}
where {\bf ch} is the characteristic map (cf. \cite{Jing91}).
\begin{cor}Let $S^{\lambda}$ be the Specht module corresponding to partition $\lambda\vdash n$ and $1\leq m<\frac{n}{2}.$ Then we have the irreducible decomposition as $\mathfrak{S}_n$-modules:
\begin{align}\label{e:decom}
V^{(n)}\otimes V^{(n-m,m)}\simeq \bigoplus_{\lambda\in DHP(n)}(N^{(n-m-\lambda_1)}(\lambda)-N^{(m-\lambda_1)}(\lambda))S^{\lambda}.
\end{align}
\end{cor}

Aokage \cite{A2} obtained the explicit irreducible decomposition of $(V^{(n)})^{\otimes 2}$ when $n$ is even, so \eqref{e:decom} offers the formula for a general tensor product. Recall that the symmetric functions $P_{\mu}(x;-1)$ are well-defined for all partitions $\mu$, so $g_{\mu\lambda}$ are defined similarly as \eqref{e:bdefn} for any partitions $\lambda, \mu$:
\begin{align}\label{e:bdefn2}
P_{\mu}(x; -1)=\sum\limits_{\lambda}g_{\mu\lambda}s_{\lambda}(x).
\end{align}

 Note that the following identities between the Schur $P$-functions and the Schur functions hold by using
 the tensor product of the spin representations of the symmetric group \cite{A}:
\begin{align}
\begin{split}\label{e:identities}
\sum_{\lambda\in HP(n)\backslash HOP(n)}s_{\lambda}(x)&=\sum_{l(\mu)\leq 2}(-1)^{\mu_2}P_{\mu}(x;-1)\\
\sum_{\lambda\in HOP(n)}s_{\lambda}(x)&=\sum_{l(\mu)= 2}(-1)^{\mu_2+1}P_{\mu}(x;-1)
\end{split}
\end{align}
where $HOP(n)\doteq\{\lambda\in HP(n)\mid \lambda_1~ \text{is odd}\}$ and $n=2r$ is even.

Aokage has conjectured the following result at the end of his paper \cite{A}.
\begin{thm}\label{t:conj} For $\lambda=(n-j,1^j)\in HP(n),$
\begin{align}
g_{(r^2)\lambda}=
\begin{cases}
0& \text{if $j<r$}\\
(-1)^{r+j}&\text{if $j\geq r$}.
\end{cases}
\end{align}
\end{thm}
As an application of our two-row formula for $b_{\xi\lambda},$ we will present a proof of Aokage's conjecture.

Combining with above two identities in \eqref{e:identities}, we have
\begin{align*}
P_{n}(x;-1)+2\sum_{i\geq1}^{r}(-1)^iP_{(n-i,i)}(x;-1)=\sum_{j=0}^{n}(-1)^js_{(n-j,1^j)}(x).
\end{align*}
Thus,
\begin{align*}
P_{(r^2)}(x;-1)=\frac{1}{4}\sum_{i\geq0}^{r-1}(-1)^{i+r+1}Q_{(n-i,i)}(x;-1)+\frac{1}{2}\sum_{j=0}^{n}(-1)^{r+j}s_{(n-j,1^j)}(x).
\end{align*}
By the orthonormality of $s_{\lambda}$,
\begin{align*}
g_{(r^2)\lambda}=\frac{1}{4}\sum_{i\geq0}^{r-1}(-1)^{i+r+1}b_{(n-i,i)\lambda}+\frac{1}{2}(-1)^{r+j}\delta_{(n-j,1^j)\lambda}.
\end{align*}
It follows from the remark below Theorem \ref{t:tworow}, we have $g_{(r^2)\lambda}=0$ unless $\lambda\in DHP(n)$. Now let's show Thm. \ref{t:conj}.

\begin{proof}
Let $\lambda=(n-j,1^j)\in HP(n),$ we have
\begin{align*}
g_{(r^2)\lambda}&=\frac{1}{2}(-1)^{r+1}+\sum_{i=1}^{r-1}(-1)^{i+r+1}(N^{(j-i)}(\lambda)-N^{(i+j-n)}(\lambda))+\frac{1}{2}(-1)^{r+j}\\
&=\frac{1}{2}(-1)^{r+1}+(-1)^{r+1}(\sum_{i=1}^{min\{r-1,j\}}(-1)^i-\sum_{i=n-j}^{r-1}(-1)^i)+\frac{1}{2}(-1)^{r+j}.
\end{align*}
Then the result follows immediately by a careful analysis of $j$ and direct computation. \end{proof}

We remark that there exists a quadratic expression of the $P$-function in terms of Schur functions \cite{LLT}.
Explicit and direct linear expansion  \eqref{e:bdefn2} in general is thus needed. Indeed, we
can give a compact formula of $g_{(r^2)\lambda}$ for any partition $\lambda$.
\begin{thm}
For $\lambda=(\lambda_1,\lambda_2,2^{m_2},1^{m_1})\in DHP(n)\backslash HP(n)$, we have that
\begin{align}
g_{(r^2)\lambda}=\sum_{i=1}^{r-1}(-1)^{i+r+1}(N^{(n-i-\lambda_1)}(\lambda)-N^{(i-\lambda_1)}(\lambda)).
\end{align}

By considering $\lambda$ case by case, we have that
\begin{align*}
g_{(r^2)\lambda}=
\begin{cases}
1& \text{if $\lambda_2+m_1-1\leq \lambda_1\leq \lambda_2+m_1+1$,}\\
0& \text{otherwise}.
\end{cases}
\end{align*}
\end{thm}

\vskip 0.5 in
{\bf Tables for $K^-_{\xi\mu}(t)$, $2\leq n\leq 6$}\\
Here $[n]=t^{n-1}+\cdots+t+1,$ $[n]!!=[n][n-2]\cdots$.
For completeness, we include $n=2, 3, 4$ from \cite{W}.

\begin{table}[H]

\caption{\label{tab:2}n=2}

 \begin{tabular}{|c|c|c|}

  \hline

 \tabincell{c}{$\mu\backslash \xi$ } & $(2)$  \\

  \hline

$(2)$ & \tabincell{c}{$2$}  \\

\hline

$(1^2)$ & \tabincell{c}{$2[2]$} \\

\hline

 \end{tabular}

\end{table}

\begin{table}[H]

 \centering

\caption{\label{tab:3}n=3}

 \begin{tabular}{|c|c|c|c|}

  \hline

 \tabincell{c}{$\mu\backslash \xi$ } & $(3)$ & $(2,1)$ \\

  \hline

$(3)$ & \tabincell{c}{$2$} & \tabincell{c}{$0$}  \\

\hline

$(2,1)$ & \tabincell{c}{$2[2]$} & \tabincell{c}{$4$}  \\

\hline
$(1^3)$   & \tabincell{c}{$2[4]$} 
& \tabincell{c}{$4t[2]$}      \\

  \hline

 \end{tabular}

\end{table}

\begin{table}[H]

 \centering

\caption{\label{tab:4}n=4}

 \begin{tabular}{|c|c|c|c|}

  \hline

 \tabincell{c}{$\mu\backslash \xi$ } & $(4)$ & $(3,1)$  \\

  \hline

$(4)$ & \tabincell{c}{$2$} & \tabincell{c}{$0$}  \\

\hline

$(3,1)$ & \tabincell{c}{$2[2]$} & \tabincell{c}{$4$}   \\

\hline

$(2^2)$ & \tabincell{c}{$2t[2]$}  & \tabincell{c}{$4[2]$} \\

  \hline
$(2,1^2)$  & \tabincell{c}{$2[4]$}    & \tabincell{c}{$4[2]^2$}  \\

  \hline
$(1^4)$   & \tabincell{c}{$2[6]!!/[3]!$}    & \tabincell{c}{$4t[4]!!$}    \\
  \hline

 \end{tabular}

\end{table}

\begin{table}[H]

 \centering

\caption{\label{tab:5}n=5}

 \begin{tabular}{|c|c|c|c|c|}

  \hline

 \tabincell{c}{$\mu\backslash \xi$ } & $(5)$ & $(4,1)$ & $(3,2)$  \\

  \hline

$(5)$ & \tabincell{c}{$2$} & \tabincell{c}{$0$} & \tabincell{c}{$0$} \\

\hline

$(4,1)$ & \tabincell{c}{$2[2]$} & \tabincell{c}{$4$} & \tabincell{c}{$0$}  \\

\hline

$(3,2)$ & \tabincell{c}{$2t[2]$}  & \tabincell{c}{$4[2]$}  & \tabincell{c}{$4$}  \\

\hline
$(3,1^2)$ & \tabincell{c}{$2[4]$}  & \tabincell{c}{$4[2]^2$}  & \tabincell{c}{$4[2]$}  \\

  \hline
$(2^2,1)$  & \tabincell{c}{$2t[4]$}    & \tabincell{c}{$4[2][3]$}    & \tabincell{c}{$4[2]^2$}   \\

  \hline
$(2,1^3)$  & \tabincell{c}{$2[6]!!/[3]!$}    & \tabincell{c}{$4[4][3]$}    & \tabincell{c}{$4t[2]([3]+1)$}    \\

  \hline
$(1^5)$   & \tabincell{c}{$2[8]!!/[4]!$}    & \tabincell{c}{$4t[6]!!/[2]$}    & \tabincell{c}{$4t^2[4]^2$}     \\
  \hline

 \end{tabular}

\end{table}

\begin{table}[H]

 \centering

\caption{\label{tab:6}n=6}

 \begin{tabular}{|c|c|c|c|c|c|}

  \hline

 \tabincell{c}{$\mu\backslash \xi$ } & $(6)$ & $(5,1)$ & $(4,2)$ & $(3,2,1)$ \\

  \hline

$(6)$ & \tabincell{c}{$2$} & \tabincell{c}{$0$} & \tabincell{c}{$0$} & \tabincell{c}{$0$} \\

\hline

$(5,1)$ & \tabincell{c}{$2[2]$} & \tabincell{c}{$4$} & \tabincell{c}{$0$} & \tabincell{c}{$0$}  \\

\hline

$(4,2)$ & \tabincell{c}{$2t[2]$}  & \tabincell{c}{$4[2]$}  & \tabincell{c}{$4$}  & \tabincell{c}{$0$}   \\

\hline
$(4,1^2)$  & \tabincell{c}{$2[4]$}    & \tabincell{c}{$4[2]^2$}    & \tabincell{c}{$4[2]$}    & \tabincell{c}{$0$}   \\

  \hline
$(3,3)$ & \tabincell{c}{$2t^2[2]$}  & \tabincell{c}{$4t[2]$}  & \tabincell{c}{$4[2]$}  & \tabincell{c}{$0$}  \\

  \hline
$(3,2,1)$  & \tabincell{c}{$2t[4]$}    & \tabincell{c}{$4[2][3]$}    & \tabincell{c}{$4[2](t+2)$}    & \tabincell{c}{$8$} \\

  \hline
$(3,1^3)$   & \tabincell{c}{$2[6]!!/[3]!$}    & \tabincell{c}{$4[4][3]$}    & \tabincell{c}{$4[2]^2[3]$}    & \tabincell{c}{$8t[2]$}   \\

\hline
$(2^3)$   & \tabincell{c}{$2t^3[4]$}    & \tabincell{c}{$4t[4]!!$}    & \tabincell{c}{$4[2]([4]+t^2)$}    & \tabincell{c}{$8t[2]$}   \\

  \hline
$(2^2,1^2)$   & \tabincell{c}{$2t[6]!!/[3]!$}    & \tabincell{c}{$4[4]^2$}    & \tabincell{c}{$4[2]^2([4]+t)$}    & \tabincell{c}{$8t[2]^2$}  \\

  \hline
$(2,1^4)$   & \tabincell{c}{$2[8]!!/[4]!$}    & \tabincell{c}{$4[4][6]!!$}    & \tabincell{c}{$4t[4]!!([4]+1)$}    & \tabincell{c}{$8t^2[4]!!$}   \\

  \hline
$(1^6)$   & \tabincell{c}{$2[10]!!/[5]!$}    & \tabincell{c}{$4t[8]!!/[3]!$}    & \tabincell{c}{$4t^2[5][6]!!/[3]$}    & \tabincell{c}{$8t^4[6]!!/[3]$}   \\
  \hline
 \end{tabular}

\end{table}

\vskip30pt \centerline{\bf Acknowledgments}
We thank Jinkui Wan and Weiqiang Wang for helpful discussions on the subject.
The project is partially supported by
the Simons Foundation under grant no. 523868 and NSFC grant 12171303.
\bigskip

\bibliographystyle{plain}

\end{document}